\documentclass[10pt]{amsart}
\usepackage{amsfonts}
\usepackage{amssymb}
\usepackage{amsmath}
\usepackage[arrow,matrix,curve]{xy}

\newcommand{\C}{{\mathbb C}}
\newcommand{\R}{{\mathbb R}}

\newcommand{\Q}{{\mathbb Q}}

\newtheorem{theorem}{Theorem}[section]

\newtheorem{corollary}[theorem]{Corollary}
\newtheorem{proposition}[theorem]{Proposition}

\begin{document}

\title{Semigroup homomorphisms on matrix algebras}
\author[B. Burgstaller]{Bernhard Burgstaller}
\address{}
\email{bernhardburgstaller@yahoo.de}

\begin{abstract}
We explore the connection between ring homomorphisms
and semigroup homomorphisms on matrix algebras over
rings or $C^*$-algebras.
%
%
\end{abstract}


\maketitle

\section{Introduction}

It is an interesting question what possibly small portion of information
distinguishes multiplicative semigroup homomorphism between rings and ring homomorphisms between rings.
Rings on which every semigroup homomorphism is automatically additive
are called to have unique addition and there exists a vast literature on this
topic, see R.E. Johnson \cite{johnson}, L.M. Gluskin \cite{gluskin}, W.S. Martindale \cite{martindale},
R.E. Peinado \cite{peinado}, A.V. Mikhalev \cite{mikhalev} and many others.
It is not possible to characterize semigroups which are the multiplicative semigroup of a ring
axiomatically, see S.R. Kogalovskij \cite{kogalovskij}. 
There is also an extensive 
investigation when the occurring rings happen
to be matrix rings of the form $M_n(R)$, 
and here one is interested in classifying all
semigroup homomorphisms between them.
Confer J. Landin and I. Reiner \cite{reiner}, M.jun. {Jodeit} and T.Y. {Lam} \cite{lam},
D. Kokol-Bukov\v sek \cite{bukovsek}, J. Marovt \cite{marovt}, X. Zhang and C. Cao \cite{zhang2},
M. {Omladi\v{c}} and B. {Kuzma} \cite{omladic}
and many others.
%
The prototype answer appears to be that every semigroup homomorphism
$\phi:GL_n(K) \rightarrow GL_m(K)$ with $K$ a division ring and $m < n$ is of the form
$$\phi (x) = \psi(\mbox{det}(x))$$
for a semigroup homomorphism $\phi:{R^*}/[R^*,R^*] \rightarrow GL_m(K)$
and Dieudonn\'e's determinant, see 
D.\v{Z}. {Djokovi\'c} \cite{djokovic}.
For integral domains $R$, $M_n(R)$ has unique addition \cite{lam}.
Most 
investigations on semigroup homomorphisms of matrix algebras have ground rings principal ideal domains, fields or division rings.

Since a ring $R$ is Morita equivalent to its matrix ring $M_n(R)$
it is often no big restriction if one considers matrix rings. 
For example $K$-theory cannot distinguish between the ring and its stablization
by matrix.
Similar things can be said for $C^*$-algebras and their notion of Morita equivalence and
topological $K$-theory.

In this short note we show that a semigroup homomorphisms $\phi:M_2(R) \rightarrow S$ for
rings
$R$ and $S$ is a ring homomorphism if and only if it satisfies the single relation $\phi(e_{11})+\phi(e_{22})= \phi(1)$.
An analogous statement holds for $C^*$-algebras.
See Proposition \ref{prop1}, Corollary \ref{corollary_cstar} and Proposition \ref{prop_last} . 

Import and much deeper related results are the classification of $*$-semigroup
endomorphisms 
on the $C^*$-algebra $B(H)$ for $H$ an 
infinite Hilbert space by J. {Moln\'ar} \cite{molnar} and
of bijective semigroup homomorphisms between
standard 
operator algebras of Banach spaces by P. {\v{S}emrl} \cite{semrl} and J. {Moln\'ar} \cite{molnar2}.
Notice that 
$B(H) \cong M_n(B(H))$ 
is matrix-stable for which our observation applies.
When finishing this note we came also accross the strongly 
related paper \cite{hakeda} by J. Hakeda,
but it considers bijective $*$-semigroup isomorphisms between $*$-algebras.

We will also 
investigate how group homomorphisms (of the form $\phi \otimes id_{M_{16}}$) 
on unitary and general linear groups
of matrix $C^*$-algebras 
can be extended to ring or $*$-homomorphisms, see
Propositions \ref{prop_unitaer} and \ref{abl}.
It seems to be an interesting and widely open question which group homomorphisms between groups $GL(M_n(R))$ for typically
noncommutative non-division rings $R$ with zero divisors even exist, if not restrictions of ring homomorphisms on $M_n(R)$.
For example $\left ( \begin{matrix} p & 1-p \\ 1-p & p \end{matrix} \right )$ is an invertible matrix for an orthogonal projection
$p \in B(H)$, but no entry is invertible 
and Dieudonn\'e's determinant is not 
applicable.

\section{Algebra homomorphisms and semigroup homomorphisms}

For a ring $A$ we shall denote $M_n(A)$ also by
$A \otimes M_n$.
For algebras $A$ latter denotes the algebra tensor product.
We write $e_{ij} \in M_n$ and $e_{ij} := 1 \otimes e_{ij} \in A \otimes M_n$
for the usual matrix units.
We also use the notation $\phi_n$ for $\phi \otimes id_{M_n}:A \otimes M_n \rightarrow
B \otimes M_n$.
A $*$-semigroup homomorphism between $C^*$-algebras means an
involution respecting semigroup homomorphism.
For unital algebras we write $\lambda$ for $\lambda 1$, where $\lambda$ is a scalar. 
We say $\phi$ is $K$-homogeneous if $\phi(\lambda x)= \lambda \phi(x)$ for all $x$ and scalars $\lambda \in K$.
%

%


\begin{proposition}    \label{prop1}
Let $A,B$ be rings where $A$ is unital and $\varphi : A \otimes M_2 \rightarrow B$ an arbitrary function.
Then the following are equivalent:
\begin{itemize}
\item[(a)]

$\varphi$ is a ring homomorphism.

\item[(b)]

$\varphi$ is a semigroup homomorphism such that 
\begin{equation}    \label{eq_e}
\varphi(1) = \varphi(e_{11}) + \varphi(e_{22}).
\end{equation}

\end{itemize}

\end{proposition}

\begin{proof}
Clearly (a) implies (b). Assume (b).
We have $\varphi(x_{ij} \otimes e_{ij}) \varphi(y_{kl} \otimes e_{kl}) = \varphi(x_{ij} y_{kl} \otimes e_{il}) \delta_{j,k}$
for $1 \le i,j,k,l \le 2$.
One has 
$$\varphi(x) = \varphi(1) \varphi(x) \varphi(1) = \sum_{i,j=1}^2 \varphi(x_{ij} \otimes e_{ij})$$
for all
$x = \sum_{i,j=1}^2 x_{ij} \otimes e_{ij} \in A \otimes M_2$
by (\ref{eq_e}).
%
Now notice that
$$\left ( \begin{matrix} 1 & a \\ 0 & 0 \end{matrix} \right ) \left ( \begin{matrix} b & 0 \\ 1 & 0 \end{matrix} \right )
=  \left ( \begin{matrix} a+b & 0 \\ 0 & 0 \end{matrix} \right )
$$
for all $a,b \in A$.
Applying here $\varphi$, using its multiplicativity and observing the upper left corner we obtain $\varphi(a \otimes e_{11})+\varphi(b \otimes e_{11}) =
\varphi(a \otimes e_{11} + b \otimes e_{11})$ and similar so for all other corners.
We conclude that $\varphi$ is additive.
\end{proof}

\begin{corollary}
Let $A$ and $B$ be rings where $A$ is unital.
Then $\varphi: A \rightarrow B$ is a ring homomorphism if and only if
$\varphi \otimes id_{M_2}$ is a semigroup homomorphism.
\end{corollary}

%
%



\begin{corollary}     \label{proposition_2}
Let $A,B$ be algebras over a field $K$ where $A$ is unital and $\varphi : M_2 \otimes A \rightarrow B$ an
arbitrary function.
Then the following are equivalent:
\begin{itemize}
\item[(a)]

$\varphi$ is an algebra homomorphism.

\item[(b)]
$\varphi$ is a semigroup homomorphism which is linear on $K e_{11} + K e_{22}$.

\end{itemize}
If $K = \C$ then we may also add

\begin{itemize}

\item[(c)] 
$\varphi$ is a semigroup homomorphism which is 
linear on $\R 1$ and satisfies 
\begin{equation}  \label{equ_i}  
\varphi(i)= i \varphi(e_{11}) + i\varphi(e_{22}).
\end{equation}

\end{itemize}

\end{corollary}

\begin{proof}
(a) $\Rightarrow$ (b) $\Rightarrow$ (c) 
are trivial.
(b) $\Rightarrow$ (a): $\varphi$ is $K$-homogeneous and by Proposition \ref{prop1} a ring homomorphism.
(c) $\Rightarrow$ (a):
$\varphi$ is $\R$-linear by $\varphi(\lambda 1 x)= \lambda \varphi(1) \varphi(x)$ ($\lambda \in \R$). 
If we take (\ref{equ_i}) to the four then we get
(\ref{eq_e}). 
Combining (\ref{equ_i}) and (\ref{eq_e}) gives $\varphi(i) = i \varphi(1)$ and thus $\varphi$
is $\C$-linear. The assertion follows then from 
Proposition \ref{prop1}.
\end{proof}

\begin{corollary}     \label{corollary_cstar}
Let $A,B$ be $C^*$-algebras where $A$ is unital and $\varphi : M_2 \otimes A \rightarrow B$ an
arbitrary function.
Then the following are equivalent:
\begin{itemize}
\item[(a)]

$\varphi$ is a $*$-homomorphism.



\item[(b)]
$\varphi$ is a $*$-semigroup homomorphism satisfying identity
(\ref{equ_i}).

\end{itemize}

\end{corollary}

\begin{proof}
Taking (\ref{equ_i}) to the four yields
(\ref{eq_e}). Hence $\varphi$ is a ring homomorphism by Proposition \ref{prop1} and thus $\Q$-linear.
By \cite[Theorem 3.7]{tomforde} $\varphi$ is continuous and thus $\C$-linear
by combining (\ref{equ_i}) and (\ref{eq_e}).
\end{proof}

\begin{corollary}
Let $A$ and $B$ be $C^*$-algebras where $A$ is unital.
Then $\varphi: A \rightarrow B$ is a $*$-homomorphism 
if and only if
$\varphi \otimes id_{M_2}$ is a $*$-semigroup homomorphism and $\varphi(i)= i \varphi(1)$.
\end{corollary}

\begin{proof}
By Proposition \ref{prop1} $\varphi \otimes id_{M_2}$ is a ring homomorphism,
thus $\Q$-linear and continuous by \cite[Theorem 3.7]{tomforde}.
\end{proof}

\begin{proposition}   \label{prop22}
Let $A,B$ be $C^*$-algebras where $A$ is unital and $\varphi : \overline{GL(M_2 \otimes A)} \rightarrow B$
an arbitrary function (norm closure).
Then the following are equivalent:
\begin{itemize}
\item[(a)]

$\varphi$ extends to a $*$-homomorphism $M_2 \otimes A \rightarrow B$.

\item[(b)]
$\varphi$ is a $*$-semigroup homomorphism 
satisfying identity (\ref{equ_i}).

\end{itemize}

Similarly, $\varphi$ extends to an algebra homomorphism if and ony if
$\varphi$ is a semigroup homomorphism which 
is continuous on $\R 1$ and satisfies identity
(\ref{equ_i}).


\end{proposition}

\begin{proof}
Assume that $\varphi$ is a semigroup homomorphism satisfying (\ref{equ_i}).
Consider 
the matrices
$$ \gamma_c :=\left ( \begin{matrix} c & \lambda \\ \lambda & 0 \end{matrix} \right ), \; \alpha_a :=\left ( \begin{matrix} 1 & a \\ 0 & \lambda \end{matrix} \right ), \;
\beta_b := \; \left ( \begin{matrix} b & \lambda \\ 1 & 0 \end{matrix} \right )
$$
for $\lambda \in \R \backslash\{0\}$ and $a,b,c \in A$.
They are invertible; 
just notice that they are
evidently bijective 
operators on $H \oplus H$ for a representation of $A$ on a Hilbert space $H$.
Letting $\lambda \rightarrow 0$ we see that all single matrix entries $a_{ij} \otimes e_{ij}$ (for all $a_{ij} \in A$)
and all matrices indicated in the proof of Proposition \ref{prop1} are in 
$\overline{GL(M_2 \otimes A)}$.
Taking (\ref{equ_i}) to the four yields identity (\ref{eq_e}).
By the proof of Proposition \ref{prop1} we see that
$\varphi(a) = \sum_{i,j=1}^n \varphi(a_{ij} \otimes e_{ij})$ for all $a \in \overline{GL(M_2 \otimes A)}$, which we use now as a definition for $\varphi$
for all $a \in A \otimes M_2$.
Also by the proof of Proposition \ref{prop1} we have $\varphi(a_{ij} \otimes e_{ij} + b_{ij} \otimes e_{ij})
= \varphi(a_{ij} \otimes e_{ij}) + \varphi(b_{ij} \otimes e_{ij})$ for all $a_{ij},b_{ij} \in A$,
which shows that the extended $\varphi$ is additive.


Since every element in a $C^*$-algebra can be written as a finite sum of invertible
elements of the form $\lambda u$ with $\lambda \in \C$ and $u$ unitary (\cite{murphy}) we see that $\varphi$ is multiplicative.
If the originally given $\varphi$ respects involution this also shows that the new $\varphi$ does so;
in this case we are done with Corollary \ref{corollary_cstar}.

Otherwise, for proving the second equivalence we proceed: 
Since $\varphi$ is a ring homomorphism it is $\Q$-linear.
By continuity $\varphi(\lambda 1)= \lambda \varphi(1)$ for all $\lambda \in \R$
and for $\lambda=i$ by (\ref{eq_e}) and (\ref{equ_i}).
Hence $\varphi$ is $\C$-linear. 
%
%
\end{proof}

\section{$C^*$-homomorphisms and group homomorphisms}

The methods of this section applies analogously to rings $A$ and $B$, or Banach algebras where we use topology,
if every element in such rings allows to be written as a finite sum of invertible elements.
This is true for $C^*$-algebras (\cite{murphy}).


\begin{proposition}   \label{prop_unitaer}
Let $\varphi:A \rightarrow B$ be an arbitrary function between unital $C^*$-algebras
$A$ and $B$.

Then $\varphi$ is a unital $*$-homomorphism if and only if
$\varphi$ is $\C$-homogeneous
and
$\varphi \otimes id_{M_{16}}$ restricts to a group homomorphism
$$U(A \otimes M_{16}) \rightarrow U(B \otimes M_{16}) .$$
\end{proposition}

\begin{proof}
Since $\varphi_{16}$ is unital, necessarily $\varphi$ is unital.
Embedding $U(M_n(A)) \subseteq U(M_{16}(A))$ via $u \mapsto \mbox{diag}(u,1)$
it is clear that $\varphi_n$ restricts to group homomorphims
between the unitary groups too for $1 \le n \le 16$.
As $\varphi(u) \varphi(u^*)=1$ for $u \in U(A)$, $\varphi(u^*)= \varphi(u)^*$.


To simplify notation, let us say that $a$ is a 
scaled unitary in $A$
if $a \in \C U(A)$. 
The set of scaled unitaries forms a monoid.
Since $\varphi$ is $\C$-homogeneous, the maps $\varphi \otimes id_{M_n}$
restrict also to
monoid
homomorphisms $\C U(A \otimes M_n) \rightarrow
\C U(B \otimes M_n)$
between the set of 
scaled unitaries.

Let $a,b$ be scaled unitaries in $A$. 
Define scaled unitaries
\begin{equation}   \label{matrix1a}
u = 
\left (   \begin{matrix} 1 & a \\ -a^* & 1 \end{matrix} \right ), \qquad
v = 
\left (  \begin{matrix} b & -1 \\ 1 & b^* \end{matrix} \right ) .
\end{equation}
Their product is the scaled unitary
$$u v = 
\left (  \begin{matrix} a+b & -1+ ab^* \\ 1 -b a^*  & a^* + b^* \end{matrix} \right ) .$$
Applying here $\varphi \otimes id$ using $\varphi_2(uv)= \varphi_2(u) \varphi_2(v)$ and observing the upper left corner we obtain
$\varphi(a+b) = \varphi(a) + \varphi(b)$.

Now call the product $uv$ $a'$ and define
$b'$ to be also the product $uv$, however, with $a$ and $b$
replaced by other scaled unitaries $c$ and $d$, respectively, in $A$.


Now consider the same matrices $u$ and $v$ as in (\ref{matrix1a}) above, but
with $a$ replaced by $a'$ and $b$ 
by $b'$.
These are four times four matrices. Consider again the product $u v$ of these newly defined matrices $u$ and $v$.
It has $a' + b'$ in the $2 \times 2$ upper right corner and so the entry
$a+b+c+d$
in the $1 \times 1$ upper right corner.
Applying $\varphi \otimes id_{M_4}$ to this identity of $4 \times 4$-matrices 
and using
$$\varphi_4(uv) = \varphi_4(u) \varphi_4(v) = 
\left (   \begin{matrix} \varphi_2(1) & \varphi_2(a') \\ \varphi_2(-a'^*) & \varphi_2(1) \end{matrix}
\right )
\left (  \begin{matrix} \varphi_2(b') & \varphi_2(-1) \\ \varphi_2(1) & \varphi_2(b'^*) \end{matrix} \right )$$
yields
$\varphi(a+b +c+d) = \varphi(a+b) + \varphi(c+d) = \varphi(a) + \varphi(b) + \varphi(c)
+\varphi(d)$
by comparing the upper left corner.

Repeating this recursive procedure 
two more times
we get additivity of $\varphi$ of 
sixteen scaled unitaries.
Since we may write any element of $A$ as the sum of four scaled unitaries (\cite{murphy})
it is 
obvious that $\varphi$ is a $*$-homomorphism.
\end{proof}

\begin{proposition}   \label{abl}
Let $\varphi:A \rightarrow B$ be an arbitrary function between unital $C^*$-algebras
$A$ and $B$.

Then $\varphi$ is 
a unital ring homomorphism if and only if
$\varphi \otimes id_{M_{16}}$ restricts to a group homomorphism
$$GL(A \otimes M_{16}) \rightarrow GL(B \otimes M_{16}) .$$
\end{proposition}

\begin{proof}
We proof this exactly by the same recursive procedure as in the last proof. All we have to do is to replace scaled unitaries by invertible elements.
The zero element we do not split up:
for example we write $\varphi(a+0+c+d)= \varphi(a+0)+ \varphi(c)+ \varphi(d)$;
this shows that $\varphi$ is additive for a sum of up to $16$ unitaries.
\end{proof}

\begin{proposition}    \label{prop_last}
Let $\varphi:GL(A \otimes M_2) \rightarrow B$ be an arbitrary function where
$A$ and $B$ are $C^*$-algebras and $A$ is unital.

Then $\varphi$ extends to a $*$-homomorphism $A \otimes M_2 \rightarrow B$ if and only if
$\varphi$ is a $*$-semigroup homomorphism
which is uniformly continuous 
and satisfies the identity
$$\varphi \left ( \begin{matrix} i & 0 \\ 0 & i \end{matrix} \right )
= \lim_{\lambda \rightarrow 0, \, \lambda \in \R_+}
 i \varphi \left ( \begin{matrix} 1 & \lambda \\ \lambda & 0 \end{matrix} \right )
+
 i \varphi \left ( \begin{matrix} 0 & \lambda \\ \lambda & 1 \end{matrix} \right ) .
$$

An analogous equivalence holds true without the star 
prefixes (omitting $*$-).
%
\end{proposition}

\begin{proof}
Since $\varphi$ is uniformly continuous 
it is clear that it maps Cauchy sequences to Cauchy sequences.
We can thus extend it continuously to $\overline{GL(A \otimes M_2)}$ via the limits of Cauchy sequences.
The assertion follows then from Proposition \ref{prop22}.
\end{proof}

\bibliographystyle{plain}
\bibliography{references}

\begin{thebibliography}{10}

\bibitem{gluskin}
L.M. {Gluskin}.
\newblock {Semigroups and rings of endomorphisms of linear spaces. I.}
\newblock {\em {Transl., Ser. 2, Am. Math. Soc.}}, 45:105--137, 1965.

\bibitem{hakeda}
Josuke {Hakeda}.
\newblock {Additivity of *-semigroup isomorphisms among *-algebras.}
\newblock {\em {Bull. Lond. Math. Soc.}}, 18:51--56, 1986.

\bibitem{lam}
M.jun. {Jodeit} and T.Y. {Lam}.
\newblock {Multiplicative maps of matrix semi-groups.}
\newblock {\em {Arch. Math.}}, 20:10--16, 1969.

\bibitem{johnson}
R.E. {Johnson}.
\newblock {Rings with unique addition.}
\newblock {\em {Proc. Am. Math. Soc.}}, 9:57--61, 1958.

\bibitem{kogalovskij}
S.R. {Kogalovskij}.
\newblock {On multiplicative semigroups of rings.}
\newblock {\em {Sov. Math., Dokl.}}, 2:1299--1301, 1961.

\bibitem{bukovsek}
Damjana {Kokol-Bukov\v sek}.
\newblock {Matrix semigroup homomorphisms into higher dimensions.}
\newblock {\em {Linear Algebra Appl.}}, 420(1):34--50, 2007.

\bibitem{reiner}
Joseph {Landin} and Irving {Reiner}.
\newblock {Automorphisms of the general linear group over a principal ideal
  domain.}
\newblock {\em {Ann. Math. (2)}}, 65:519--526, 1957.

\bibitem{marovt}
Janko {Marovt}.
\newblock {Homomorphisms of matrix semigroups over division rings from
  dimension two to four.}
\newblock {\em {Linear Algebra Appl.}}, 432(6):1595--1607, 2010.

\bibitem{martindale}
Wallace S.~III {Martindale}.
\newblock {When are multiplicative mappings additive.}
\newblock {\em {Proc. Am. Math. Soc.}}, 21:695--698, 1969.

\bibitem{mikhalev}
A.V. {Mikhalev}.
\newblock {Multiplicative classification of associative rings.}
\newblock {\em {Math. USSR, Sb.}}, 63(1):205--218, 1989.

\bibitem{molnar2}
Lajos {Moln\'ar}.
\newblock {On isomorphisms of standard operator algebras.}
\newblock {\em {Stud. Math.}}, 142(3):295--302, 2000.

\bibitem{molnar}
Lajos {Moln\'ar}.
\newblock {*-semigroup endomorphisms of $B(H)$.}
\newblock In {\em {Recent advances in operator theory and related topics. The
  B\'ela Sz\H{o}kefalvi-Nagy memorial volume. Proceedings of the memorial
  conference, Szeged, Hungary, August 2--6, 1999}}, pages 465--472. Basel:
  Birkh\"auser, 2001.

\bibitem{murphy}
Gerard~J. {Murphy}.
\newblock {\em {C${}\sp*$-algebras and operator theory.}}
\newblock Boston, MA etc.: Academic Press, Inc., 1990.

\bibitem{omladic}
Matja\v{z} {Omladi\v{c}} and Bojan {Kuzma}.
\newblock {A note on homomorphisms of matrix semigroup.}
\newblock {\em {Ars Math. Contemp.}}, 6(2):247--252, 2013.

\bibitem{peinado}
R.E. {Peinado}.
\newblock {On semigroups admitting ring structure.}
\newblock {\em {Semigroup Forum}}, 1:189--208, 1970.

\bibitem{semrl}
Peter {\v{S}emrl}.
\newblock {Isomorphisms of standard operator algebras.}
\newblock {\em {Proc. Am. Math. Soc.}}, 123(6):1851--1855, 1995.

\bibitem{tomforde}
Mark {Tomforde}.
\newblock {Continuity of ring $\ast$-homomorphisms between $C^\ast$-algebras.}
\newblock {\em {New York J. Math.}}, 15:161--167, 2009.

\bibitem{djokovic}
Dragomir \v{Z}. {Djokovi\'c}.
\newblock {On homomorphisms of the general linear group.}
\newblock {\em {Aequationes Math.}}, 4:99--102, 1970.

\bibitem{zhang2}
Xian {Zhang} and Chongguang {Cao}.
\newblock {Homomorphisms between multiplicative semigroups of matrices over
  fields.}
\newblock {\em {Acta Math. Sci., Ser. B, Engl. Ed.}}, 28(2):301--306, 2008.

\end{thebibliography}

\end{document}